\documentclass[12pt,reqno]{article}

\usepackage[usenames]{color}
\usepackage{amssymb}
\usepackage{amsmath}
\usepackage{amsthm}
\usepackage{amsfonts}
\usepackage{amscd}
\usepackage{graphicx}
\usepackage{caption}

\usepackage[colorlinks=true,
linkcolor=webgreen,
filecolor=webbrown,
citecolor=webgreen]{hyperref}

\definecolor{webgreen}{rgb}{0,.5,0}
\definecolor{webbrown}{rgb}{.6,0,0}

\usepackage{color}
\usepackage{fullpage}
\usepackage{float}

\usepackage{graphics}
\usepackage{latexsym}
\usepackage{epsf}

\begin{document}

\theoremstyle{plain}
\newtheorem{theorem}{Theorem}
\newtheorem{corollary}[theorem]{Corollary}
\newtheorem{lemma}[theorem]{Lemma}
\newtheorem{proposition}[theorem]{Proposition}

\theoremstyle{definition}
\newtheorem{definition}[theorem]{Definition}
\newtheorem{example}[theorem]{Example}
\newtheorem{conjecture}[theorem]{Conjecture}
\newtheorem{observation}[theorem]{Observation}

\theoremstyle{remark}
\newtheorem{remark}[theorem]{Remark}

\begin{center}
\vskip 1cm{\LARGE\bf
Square-Weighted Zero-Sum Constants}
\vskip 1cm
\large
Krishnendu Paul and Shameek Paul \\
School of Mathematical Sciences \\
Ramakrishna Mission Vivekananda Educational and Research Institute \\
P.O. Belur Math, Dist.\ Howrah, West Bengal 711202 \\
India \\

\href{mailto:krishnendu.p.math18@gm.rkmvu.ac.in}{\tt krishnendu.p.math18@gm.rkmvu.ac.in} \\
\href{mailto:shameek.rkmvu@gmail.com}{\tt shameek.rkmvu@gmail.com}
\end{center}

\vskip .2 in

\begin{abstract}
Let $A\subseteq \mathbb Z_n$ be a subset. A sequence $S=(x_1,\ldots,x_k)$ in $\mathbb Z_n$ is said to be an $A$-weighted zero-sum sequence if there exist $a_1,\ldots,a_k\in A$ such that $a_1x_1+\cdots+a_kx_k=0$. By a square, we shall mean a non-zero square in $\mathbb Z_n$. We determine the smallest natural number $k$, such that every sequence in $\mathbb Z_n$ whose length is $k$, has a square-weighted zero-sum subsequence. We also determine the smallest natural number $k$, such that every sequence in $\mathbb Z_n$ whose length is $k$, has a square-weighted zero-sum subsequence whose terms are consecutive terms of the given sequence.
\end{abstract}

\section{Introduction} 

For a finite set $A$, we let $|A|$ denote the number of elements of $A$. For $a,b\in\mathbb Z$, we let $[a,b]$ denote the set $\{x\in\mathbb Z:a\leq x\leq b\}$.

Let $R$ be a ring with unity, $M$ be an $R$-module, and $A\subseteq R$. A subsequence $T$ of a sequence $S=(x_1,x_2,\ldots, x_{k})$ in $M$ is called an {\it $A$-weighted zero-sum subsequence} if the set $J=\{i : x_i\in T\}$ is non-empty and for every $i\in J$ there exists $a_i\in A$ such that $\sum_{i\in J} a_i x_i = 0$.

For a finite $R$-module $M$ and $A \subseteq R$, the {\it $A$-weighted Davenport constant of $M$} denoted by $D_A$ is defined to be the least positive integer $k$, such that every sequence in $M$ whose length is $k$, has an $A$-weighted zero-sum subsequence.

Adhikari and Chen \cite{AC} introduced this constant for the ring $R=\mathbb Z$, i.e., for abelian groups. We define the constant $C_A$ to be the least positive integer $k$, such that every sequence in $M$ whose length is $k$, has an $A$-weighted zero-sum subsequence whose terms are consecutive terms.

\begin{remark}\label{c}
Mondal, K. Paul, and S. Paul \cite{SKS} have shown that $D_A\leq C_A\leq |M|$.
\end{remark}

We also denote the ring $\mathbb Z/n\mathbb Z$ by $\mathbb Z_n$. Let $U(n)$ denote the group of units in $\mathbb Z_n$ and $U(n)^2$ denote the set $\{x^2 : x\in U(n)\}$. For an odd prime $p$ we denote the set $U(p)^2$ by $Q_p$. For a divisor $m$ of $n$, the homomorphism $f_{n,m}:\mathbb Z_n\to \mathbb Z_m$ is given by $f_{n,m}(a+n\mathbb Z)=a+m\mathbb Z$. Mondal et al.\ \cite[Lem.\ 7]{SKS2} showed that the image of $U(n)$ under $f_{n,m}$ is $U(m)$.

Let $p$ be a prime divisor of $n$. We say that $v_p(n)=r$ if $p^r\mid n$ and $p^{r+1}\nmid n$. Suppose $r=v_p(n)$. For every $x\in\mathbb Z_n$ we denote the image of $x$ under $f_{n,\,p^r}$ by $x^{(p)}$. Given a sequence $S=(x_1,\ldots,x_l)$ in $\mathbb Z_n$, we get a sequence $S^{(p)}=(x_1^{(p)},\ldots,x_l^{(p)})$ in $\mathbb Z_{p^r}$. From this point onwards, we will only consider the case when  $M=R=\mathbb Z_n$.

Let $\Omega(n)$ denote the number of prime factors of $n$ counted with multiplicity. Grynkiewicz and Hennecart \cite{GH} showed that $D_{U(n)^2}\geq 2\Omega(n)+\text{min}\{v_3(n),v_5(n)\}+1$ when $n$ is odd, with equality if either $3\nmid n$ or $v_3(n)\geq v_5(n)$. This extends a result of Chintamani and Moriya \cite{CM}, and another of Adhikari, David, and Urroz \cite{ADU}. These results lead quite naturally to the question of the value of $D_{S(n)}$ where $S(n)=\{\,x^2:x\in \mathbb Z_n\}\setminus \{0\}$.

We determine the value of $D_{S(n)}$ for every $n$ and show that it depends on the parity of $n$ when $n$ is a square, and on the parity of $v_2(n)$ when $n$ is not a square. We also investigate the value of $C_{S(n)}$.

We show that $C_{U(25)^2}=9$, adding to the results which were obtained by Mondal et al.\ \cite{SKS}. Using this fact, we get that $C_{S(n)}\leq 9$ when $n$ is an odd square. In this article, we have obtained the following results:

\begin{itemize} 
\item We determine the size of $S(n)$ for every $n$.

\item When $n$ is a square, we get that $D_{S(n)}=4$ or $5$ when $n$ is even or odd respectively. 

\item When $n$ is not a square, we get that $D_{S(n)}=2$ or $3$ when $v_2(n)$ is odd or even respectively.

\item When $n$ is not a square of an odd number, we get that $C_{S(n)}=D_{S(n)}$.

\item When $n$ is a square of an odd, squarefree number, we get that $C_{S(n)}=9$. 

\item When $n$ is a square of an odd number $m$ such that $m$ is divisible by $p^2$ where $p$ is a prime which is at least seven, we get that $C_{S(n)}=D_{S(n)}$.
\end{itemize}

\begin{table}[h]
\begin{center}
\begin{tabular}{ |c|*{18}{p{4mm}}| }
 \hline
 $n$       & 2 & 3 & 4 & 5 & 6 & 7 & 8 & 9 & 10 & 11 & 12 & 13 & 14 & 15 & 16 & 17 & 18 & 19 \\
 $D_{S(n)}$& 2 & 3 & 4 & 3 & 2 & 3 & 2 & 5 & 2 & 3 & 3 & 3 & 2 & 3 & 4 & 3 & 2 & 3 \\
 \hline
\end{tabular}
\end{center}
\end{table}

\begin{table}[h]
\begin{center}
\begin{tabular}{ |c|*{18}{p{4mm}}| }
 \hline
 $n$       & 20 & 21 & 22 & 23 & 24 & 25 & 26 & 27 & 28 & 29 & 30 & 31 & 32 & 33 & 34 & 35 & 36 & 37 \\
 $D_{S(n)}$& 3 & 3 & 2 & 3 & 2 & 5 & 2 & 3 & 3 & 3 & 2 & 3 & 2 & 3 & 2 & 3 & 4 & 3 \\
 \hline
\end{tabular}
\end{center}
\caption{Values of $D_{S(n)}$ for all $n\in [2,37]$.}\label{table}
\end{table}

The only values of $n$ in the set $[2,37]$ for which $C_{S(n)}$ differs from $D_{S(n)}$ are 9 and 25. The smallest $n$ for which we have not been able to determine $C_{S(n)}$ is $81$.

\section{The size of \texorpdfstring{$S(n)$}{}}

If $n=p_1^{r_1}\cdots p_s^{r_s}$ where the $p_i$'s are distinct primes, then from the Chinese remainder theorem we get an isomorphism of rings
\[\mathbb Z/n\mathbb Z\simeq \mathbb Z/ p_1^{r_1}\mathbb Z\times\cdots\times \mathbb Z/p_s^{r_s}\mathbb Z.\]
Hence, it follows that 
\[|S(n)|=\big(\,|S(p_1^{r_1})|+1\,\big)\cdots \big(\,|S(p_s^{r_s})|+1\,\big)-1.\]

Thus, it is enough to determine the size of $S(p^r)$ where $p$ is a prime and $r$ is a positive integer. 

\begin{observation}\label{unique}
Let $p$ be a prime, $r$ be a positive integer, and $a\in \mathbb Z_{p^r}\setminus\{0\}$. Then there exists a unique $k\in [0,r-1]$ such that $a=p^ku$ where $u$ is a unit.
\end{observation}

For a real number $x$, we let $\lfloor x\rfloor$ denote the greatest integer which is at most equal to $x$.

\begin{lemma}\label{sqpr}
Let $p$ be a prime, $r$ be a positive integer, and $l=\lfloor (r-1)/2\rfloor$. Then we have that $S(p^r)=\bigcup_{\,k\in [0,\,l]}\,p^{2k}U(p^r)^2$. Also, this is a disjoint union.
\end{lemma}

\begin{proof}
Let $a\in S(p^r)$. Then there exists $b\in \mathbb Z_{p^r}\setminus \{0\}$ such that $a=b^2$. Thus, there exists $u\in U(p^r)$ and $k\in [0,r-1]$ such that $b=p^ku$. As $a=p^{2k}u^2$ and $a\neq 0$, by Observation \ref{unique} we see that $2k\in [0,r-1]$ and so $k\in [0,l]$. Hence, we see that $S(p^r)\subseteq \bigcup_{\,k\in [0,\,l]}\, p^{2k}U(p^r)^2$. It is easy to see that the reverse inclusion holds. From Observation \ref{unique} it follows that this union is a disjoint union.
\end{proof}

\begin{lemma}\label{sizes}
Let $p$ be a prime, $r$ be a natural number, and $l=\lfloor (r-1)/2\rfloor$. Then for every $k\in [0,l]$ we have that
$|\,p^{2k}\,U(p^r)^2\,|=|\,U(p^{r-2k})^2\,|$.
\end{lemma}

\begin{proof}
Let $k\in [0,l]$ and $f=f_{p^r,\,p^{r-2k}}$. For all $x,y\in\mathbb Z_{p^r}$ we see that
\[p^{2k}x=p^{2k}y\iff p^{2k}(x-y)=0\iff p^{r-2k}\mid (x-y)\iff x-y\in ker~f.\]

So we see that $\varphi:p^{2k}\,U(p^r)^2\to U(p^{r-2k})^2$ defined as $\varphi(p^{2k}x)=f(x)$ is well-defined and injective. We claim that the map $\varphi$ is also surjective. Let $y\in U(p^{r-2k})^2$. As the image of $U(p^r)$ under $f$ is $U(p^{r-2k})$, it follows that there exists $x\in U(p^r)^2$ such that $f(x)=y$. Hence, we see that $\varphi(p^{2k}x)=y$. This proves our claim. So it follows that $\varphi$ is a bijection.  
\end{proof}

The next result follows from Lemmas \ref{sqpr} and \ref{sizes}. 

\begin{theorem}
If $r$ is even \[|S(p^r)|=|U(p^r)^2|+|U(p^{r-2})^2|+\cdots+|U(p^4)^2|+|U(p^2)^2|\]
and if $r$ is odd \[|S(p^r)|=|U(p^r)^2|+|U(p^{r-2})^2|+\cdots+|U(p^3)^2|+|U(p)^2|.\]
\end{theorem}

It remains to determine the size of $U(n)^2$ when $n$ is a prime power. Let $n=p^r$ where $p$ is an odd prime and $r$ is a positive integer. Ireland and Rosen \cite[Thm.\ 2, p.\ 43]{IR} have shown that $U(n)$ is a cyclic group. So there is exactly one element of order two in $U(n)$. Thus, the kernel of the onto map $U(n)\to U(n)^2$ given by $x\mapsto x^2$ has order two. Hence, we see that $U(n)^2$ has index two in $U(n)$. So it follows that $|U(n)^2|=|U(n)|/2=p^{r-1}(p-1)/2$.

We have that $U(4)^2=U(2)^2=\{1\}$.  Let $n=2^r$ where $r$ is at least three. Ireland and Rosen \cite[Thm.\ 2', p.\ 43]{IR} have shown that $U(n)\simeq \mathbb Z_2\times \mathbb Z_{2^{r-2}}$. So there are exactly three elements of order two in $U(n)$. Thus, the kernel of the onto map $U(n)\to U(n)^2$ given by $x\mapsto x^2$ has order four. Hence, we see that $U(n)^2$ has index four in $U(n)$. So it follows that $|U(n)^2|=|U(n)|/4=2^{r-1}/4=2^{r-3}$.

\section{Some general results}

\begin{observation}\label{crt'}
Let $n=p_1^{r_1}p_2^{r_2}\ldots p_k^{r_k}$ where the $p_i$'s are distinct primes. By the Chinese remainder theorem we get an isomorphism \[\varphi:\mathbb Z/n\mathbb Z\to\mathbb Z/p_1^{r_1}\mathbb Z\times \cdots \times \mathbb Z/p_k^{r_k}\mathbb Z\] given by $\varphi(a)=(a^{(p_1)},\ldots,a^{(p_k)})$. As $\varphi$ is an isomorphism, we have that $a\in S(n)$ if and only if for every prime divisor $q$ of $n$, we have that $a^{(q)}$ is a square and there exists a prime divisor $p$ of $n$ such that $a^{(p)}\neq 0$.
\end{observation}

\begin{lemma}\label{crt}
Let $S$ be a sequence in $\mathbb Z_n$ and $p$ be a prime divisor of $n$ such that $v_p(n)=r$. Suppose the sequence $S^{(p)}$ is an $S(p^r)$-weighted zero-sum sequence. Then the sequence $S$ is an $S(n)$-weighted zero-sum sequence.
\end{lemma}

\begin{proof}
Let $S=(x_1,\ldots,x_l)$. Then we have that $S^{(p)}=(x_1^{(p)},\ldots,x_l^{(p)})$. There exist $b_1,\ldots,b_l\in S(p^r)$ such that $b_1x_1^{(p)}+\cdots+b_lx_l^{(p)}=0$. By Observation \ref{crt'} we see that for every $i\in [1,l]$ there exists $a_i\in S(n)$ such that $a_i^{(p)}=b_i$ and for each prime divisor $q$ of $n/p^r$ we have $a_i^{(q)}=0$. Let $\varphi$ be the isomorphism given by the Chinese remainder theorem as in Observation \ref{crt'}. As we get that $\varphi(a_1x_1+\cdots+a_lx_l)=0$, it follows that $a_1x_1+\cdots+a_lx_l=0$. Hence, we see that $S$ is an $S(n)$-weighted zero-sum sequence in $\mathbb Z_n$.
\end{proof}

\begin{corollary}\label{pr}
Let $p$ be a prime divisor of $n$ and $r=v_p(n)$. Then we have that $C_{S(n)}\leq C_{S(p^r)}$. 
\end{corollary}

\begin{proof}
Let $m=p^r$. Suppose $S$ is a sequence in $\mathbb Z_n$ having length $C_{S(m)}$. As $S^{(p)}$ is a sequence in $\mathbb Z_m$ having length $C_{S(m)}$, it follows that there exists a subsequence $T$ of $S$ having consecutive terms such that $T^{(p)}$ is an $S(m)$-weighted zero-sum sequence. So from Lemma \ref{crt} we see that $T$ is an $S(n)$-weighted zero-sum sequence. Hence, it follows that $C_{S(n)}\leq C_{S(m)}$.
\end{proof}

We will apply the next result later in the case when $p$ is a prime.

\begin{lemma}\label{q2a}
Let $p$ be an integer which is at least two, $r$ be an odd number, and $T$ be a sequence in $\mathbb Z_{p^r}$. Suppose the image of $T$ under $f_{p^r,\,p}$ is a $U(p)^2$-weighted zero-sum sequence. Then $T$ is an $S(p^r)$-weighted zero-sum sequence.
\end{lemma}

\begin{proof}
Let $T=(x_1,\ldots,x_k)$ be a sequence in $\mathbb Z_{p^r}$ and $T'=(x_1',\ldots,x_k')$ be the image of $T$ under $f_{p^r,\,p}$. For each $i\in [1,k]$ there exist $a_i'\in U(p)^2$ such that $a_1'x_1'+\cdots +a_k'x_k'=0$. Mondal et al.\ \cite[Lem.\ 7]{SKS2} showed that the image of $U(p^r)^2$ under $f_{p^r,\,p}$ is $U(p)^2$. So for each $i\in [1,k]$ there exists $a_i\in U(p^r)^2$ such that $f_{p^r,\,p}(a_i)=a_i'$.

Let $x=a_1x_1+\cdots+a_kx_k$. As $f_{p^r,\,p}(x)=a_1'x_1'+\cdots +a_k'x_k'=0$, it follows that $p$ divides $x$. We see that $c=p^{r-1}=(p^{(r-1)/2})^2\in S(p^r)$. As $p$ divides $x$ we get that $cx=0$. Thus, it follows that $(c\,a_1)x_1+\cdots+(c\,a_k)x_k=0$. For each $i\in [1,k]$ we see that $c\,a_i\in S(p^r)$. Hence, it follows that $T$ is an $S(p^r)$-weighted zero-sum sequence.
\end{proof}

\section{When \texorpdfstring{$n$}{} is an even square}

The next two results will be used to determine the value of $D_{S(n)}$ when $n$ is an even square. 

\begin{lemma}\label{les2}
Let $r$ be a non-zero even number. Let $S=(x_1,x_2,x_3)$ be a sequence in $U(2^r)$ whose image under $f_{2^r,\,4}$ is the sequence $(1,1,1)$. Then $S$ does not have any $S(2^r)$-weighted zero-sum subsequence. 
\end{lemma}

\begin{proof}
As the image of $U(2^r)^2$ under $f_{2^r,\,4}$ is $U(4)^2=\{1\}$ and the sequence $(1,1,1)$ in $\mathbb Z_4$ does not have any zero-sum subsequence, it follows that $S$ does not have any $U(2^r)^2$-weighted zero-sum subsequence. However, as we want to show that $S$ does not have any $S(2^r)$-weighted zero-sum subsequence, we need to modify this argument. Suppose $T$ is an $S(2^r)$-weighted zero-sum subsequence of $S$. Let $I=\{i\in [1,3]:x_i$ is a term of $T\}$.

For each $i\in I$ there exists $a_i\in S(2^r)$ such that $\sum_{i\in I} a_ix_i=0$. By Lemma \ref{sqpr} for each $i\in I$ we see that $a_i=2^{r_i}u_i$ where $r_i$ is an even number which is at most $r-2$ and $u_i\in U(2^r)^2$. So we get that $\sum_{i\in I} 2^{r_i}u_ix_i=0$. Let $r'$ be the minimum of the set $\{r_i:i\in I\}$ and $J=\{i\in I:r_i=r'\}$. Let $f=f_{2^r,\,4}$. As $r'\leq r-2$ we see that four divides $\sum_{i\in J} u_i\,x_i$ and hence $0=\sum_{i\in J}f(u_i)f(x_i)=\sum_{i\in J}1$. So we get the contradiction that the sequence $(1,1,1)$ in $\mathbb Z_4$ has a zero-sum subsequence. Hence, it follows that $S$ does not have any $S(2^r)$-weighted zero-sum subsequence.
\end{proof}

\begin{lemma}\label{les}
Let $p$ be a prime and $r$ be a non-zero even number. Let $(z_1,z_2)$ be a sequence in $U(p^2)$ whose image under $f_{p^2,\,p}$ is not a $Q_p$-weighted zero-sum sequence. Let $S=(x_1,x_2,y_1)$ be a sequence in $U(p^r)$ whose image under $f_{p^r,\,p^2}$ is the sequence $(z_1,z_2,p)$. Then $S$ does not have any $S(p^r)$-weighted zero-sum subsequence.
\end{lemma}

\begin{proof}
Suppose the sequence $S=(x_1,x_2,y_1)$ has an $S(p^r)$-weighted zero-sum subsequence $T$. Let $I=\{i\in [1,2]:x_i~\textrm{is a term of}~T\}$. Let $J=\{1\}$ if $y_1$ is a term of $T$ and let $J=\emptyset$ if $y_1$ is not a term of $T$. Then for every $i\in I$ and $j\in J$ there exist $a_i,b_j\in S(p^r)$ such that $\sum_{i\in I}a_ix_i+\sum_{j\in J}b_jy_j=0$. As $y_1$ maps to $p$ under $f_{p^r,\,p^2}$ there exists $w_1\in U(p^r)$ such that $y_1=pw_1$. By Lemma \ref{sqpr}, for each $i\in I$ and $j\in J$ we see that $a_i=p^{r_i}u_i$ and $b_j=p^{s_j}v_j$ where $r_i,s_j\in [0,r-2]$ are even and $u_i,v_j\in U(p^r)^2$. So we get that
\begin{equation}\label{zs'}
\sum_{i\in I}p^{r_i}u_i\,x_i+\sum_{j\in J}p^{s_j+1}v_jw_j=0.                                                                                                                                                                                                                                                                                                                                                                                                                                                                           \end{equation}

Consider the set $L=\{r_i:i\in I\}\cup \{s_j+1:j\in J\}$. Let $r'$ be the minimum of $L$. As $r\geq 2$ is even and $s_1$ is even, it follows that $r'\leq r-1$. Suppose there exists $i\in I$ such that $r_i=r'$. We claim that $I=\{1,2\}$ and $r_1=r_2=r'$. If not, from (\ref{zs'}) we get that $p^{r'+1}$ divides $p^{r'}w$ where $w$ is a unit. As $r'\leq r-1$ we get the contradiction that $p$ divides $w$. By a similar argument, we see that $s_1+1\neq r'$.

As $r'\leq r-1$, from (\ref{zs'}) we see that $p$ divides $u_1x_1+u_2x_2$. Let $f=f_{p^r,\,p}$. We get that $f(u_1)f(x_1)+f(u_2)f(x_2)=0$. As $u_1,u_2\in U(p^r)^2$, it follows that $f(u_1),f(u_2)\in Q_p$. So the sequence $\big(f(x_1),f(x_2)\big)$ is a $Q_p$-weighted zero-sum sequence. Thus, we get the contradiction that the image of the sequence $(z_1,z_2)$ under $f_{p^2,\,p}$ is a $Q_p$-weighted zero-sum sequence. Hence, it follows that $S$ does not have any $S(p^r)$-weighted zero-sum subsequence. 
\end{proof}

\begin{theorem}\label{lbe2}
Let $n$ be an even square. Then we have that  $D_{S(n)}\geq 4$.
\end{theorem}

\begin{proof}
Mondal et al.\ \cite[Cor.\ 2, Lem.\ 7]{SKS2} have shown that for every odd prime $p$ we can find a sequence $(u_p,v_p)$ in $U(p^2)$ whose image under $f_{p^2,\,p}$ is not a $Q_p$-weighted zero-sum sequence. Consider the sequence $(u_p,v_p,p)$ in $\mathbb Z_{p^2}$.

For each prime divisor $p$ of $n$, if $n_p=p^{v_p(n)}$ then the map $f_{n_p,\,p^2}$ is onto. So by the Chinese remainder theorem we can find a sequence $S=(x_1,x_2,x_3)$ in $\mathbb Z_n$ such that for every prime divisor $p$ of $n$ the image of $S$ under $f_{n,\,p^2}$ is $(u_p,v_p,p)$ if $p$ is odd, and under $f_{n,\,4}$ is $(1,1,1)$.

For every prime divisor $p$ of $n$, we see that the sequence $S^{(p)}$ in $\mathbb Z_{p^r}$  has the form as in the statement of Lemma \ref{les2} if $p=2$, or of Lemma \ref{les} if $p$ is odd. So for every prime divisor $p$ of $n$, if $r=v_p(n)$, it follows that the sequence $S^{(p)}$ does not have any $S(p^r)$-weighted zero-sum subsequence.

Suppose $T$ is an $S(n)$-weighted zero-sum subsequence of $S$. Let $x$ be a term of $T$ and $a\in S(n)$ be the coefficient of $x$ in an $S(n)$-weighted zero-sum which is obtained from $T$. As $a\neq 0$, there is a prime divisor $p$ of $n$ such that  $a^{(p)}\neq 0$. So we get the contradiction that the sequence $S^{(p)}$ in $\mathbb Z_{p^r}$ has an $S(p^r)$-weighted zero-sum subsequence where  $r=v_p(n)$.

Thus, it follows that the sequence $S$ does not have any $S(n)$-weighted zero-sum subsequence. Hence, we see that $D_{S(n)}\geq 4$.  
\end{proof}

\begin{lemma}\label{s2a}
Let $r$ be a non-zero even number and $p$ be a positive integer. Suppose $T$ is a sequence in $\mathbb Z_{p^r}$ whose image under $f_{p^r,\,p^2}$ is a $U(p^2)^2$-weighted zero-sum sequence. Then $T$ is an $S(p^r)$-weighted zero-sum sequence.
\end{lemma}

\begin{proof}
Let $T=(x_1,\ldots,x_k)$ be a sequence in $\mathbb Z_{p^r}$ and $T'=(x_1',\ldots,x_k')$ be the image of $T$  under $f_{p^r,\,p^2}$. Suppose $T'$ is a $U(p^2)^2$-weighted zero-sum sequence. Then there exist $a_1',\ldots,a_k'\in U(p^2)^2$ such that $a_1'x_1'+\cdots +a_k'x_k'=0$. Mondal et al.\ \cite[Lem.\ 7]{SKS2} have shown that $f_{p^r,\,p^2}\big(U(p^r)^2\big)=U(p^2)^2$. Thus, for each $i\in [1,k]$ there exists $a_i\in U(p^r)^2$ such that $f_{p^r,\,p^2}(a_i)=a_i'$.

Let $x=a_1x_1+\cdots+a_kx_k$. As $f_{p^r,\,p^2}(x)=a_1'x_1'+\cdots +a_k'x_k'=0$ it follows that $p^2$ divides $x$. As $r$ is an even number which is at least two, we see that $p^{r-2}=(p^{(r-2)/2})^2\in S(p^r)$. As $p^2$ divides $x$ we get $p^{r-2}x=0$ and so $(c\,a_1)x_1+\cdots+(c\,a_k)x_k=0$ where $c=p^{r-2}$. As the $a_i$'s are in $U(p^r)^2$ it follows that $T$ is an $S(p^r)$-weighted zero-sum sequence. 
\end{proof}

The next result follows immediately from Lemma \ref{s2a}. 

\begin{corollary}\label{u2s}
Let $r$ be a non-zero even number and $p$ be a positive integer. Then we have that $D_{S(p^r)}\leq D_{U(p^2)^2}$ and $C_{S(p^r)}\leq C_{U(p^2)^2}$. 
\end{corollary}

\begin{theorem}\label{ube2'}
Let $r$ be a non-zero even number. Then we have $C_{S(2^r)}\leq 4$.  
\end{theorem}

\begin{proof}
Mondal et al.\ \cite[Cor.\ 1]{SKS} have shown that $C_{\{1\}}=4$. As $U(4)^2=\{1\}$, from Corollary \ref{u2s} it follows that $C_{S(2^r)}\leq 4$.
\end{proof}

\begin{corollary}
Let $n$ be an even square. Then we have $D_{S(n)}=C_{S(n)}=4$.
\end{corollary}

\begin{proof}
From Theorem \ref{lbe2} we have $D_{S(n)}\geq 4$. By Theorem \ref{ube2'} and Corollary \ref{pr} we have $C_{S(n)}\leq 4$. As $D_A(n)\leq C_A(n)$ for every $A\subseteq\mathbb Z_n$, it follows that $D_{S(n)}=C_{S(n)}=4$.
\end{proof}

\section{When \texorpdfstring{$n$}{} is not a square}

\begin{proposition}\label{seq}
Let $n$ be odd. We can find a sequence $S=(u,v)$ in $U(n)$ such that for each prime divisor $p$ of $n$, the image of $S$ under $f_{n,\,p}$ does not have any $Q_p$-weighted zero-sum subsequence. 
\end{proposition}

\begin{proof}
Let $p$ be a prime divisor of $n$ and $v_p(n)=r$. By \cite[Cor.\ 2]{SKS2} there exist $u_p,v_p\in U(p)$ such that the sequence $(u_p,v_p)$ does not have any $Q_p$-weighted zero-sum subsequence. As the image of $U(p^r)$ under $f_{p^r,\,p}$ is $U(p)$, there exist $u'_p,v'_p\in U(p^r)$ such that the image of the sequence $(u'_p,v'_p)$ under $f_{p^r,\,p}$ is $(u_p,v_p)$.

By the Chinese remainder theorem, there exist $u,v\in U(n)$ such that for each prime divisor $p$ of $n$ if $n_p=p^{v_p(n)}$, then the image of the sequence $S=(u,v)$ under $f_{n,n_p}$ is $(u'_p,v'_p)$. It follows that the image of $S$ under $f_{n,\,p}$ is $(u_p,v_p)$ which is the same as the image of $(u'_p,v'_p)$ under the map $f_{n_p,\,p}$.
\end{proof}

\begin{lemma}\label{lboo}
Let $p$ be an odd prime and $r$ be a positive integer. Suppose $S=(v_1,v_2)$ is a sequence in $U(p^r)$ such that the image of $S$ under $f_{p^r,\,p}$ is not a $Q_p$-weighted zero-sum sequence. Then $S$ does not have any $S(p^r)$-weighted zero-sum subsequence.
\end{lemma}

\begin{proof}
Suppose $S$ is an $S(p^r)$-weighted zero-sum sequence. Then there exist $a_1,a_2\in S(p^r)$ such that $a_1v_1+a_2v_2=0$. By Lemma \ref{sqpr} we see that there exist $u_1,u_2\in U(p^r)^2$ and even $r_1,r_2\in [0,r-1]$ such that $a_1=p^{r_1}u_1$ and $a_2=p^{r_2}u_2$. So we get that $p^{r_1}u_1v_1+p^{r_2}u_2v_2=0$. By Observation \ref{unique} we see that $r_1=r_2$ and so $p^{r_1}(u_1v_1+u_2v_2)=0$. As $r_1<r$, it follows that $p$ divides $u_1v_1+u_2v_2$. 

Let $f$ be the map $f_{p^r,\,p}$. We get that $f(u_1)f(v_1)+f(u_2)f(v_2)=0$. As $u_1,u_2\in U(p^r)^2$, it follows that $f(u_1),f(u_2)\in Q_p$. Thus, we get the contradiction that the image of $S$ under $f_{p^r,\,p}$ is a $Q_p$-weighted zero-sum sequence. Hence, it follows that $S$ is not an $S(p^r)$-weighted zero-sum sequence. As $v_1,v_2\in U(p^r)$, we see that $S$ does not have any $S(p^r)$-weighted zero-sum subsequence of length one.
\end{proof}

\begin{theorem}\label{lbo}
Let $n$ be an odd number. Then we have that $D_{S(n)}\geq 3$.
\end{theorem}

\begin{proof}
By Proposition \ref{seq} there exists a sequence $S=(u,v)$ in $U(n)$ such that for each prime divisor $p$ of $n$, the image of $S$ under $f_{n,\,p}$ does not have any $Q_p$-weighted zero-sum subsequence. Suppose $T$ is an $S(n)$-weighted zero-sum subsequence of $S$. As the terms of $S$ are in $U(n)$, we see that $T$ must be $S$. Thus, there exist $a,b\in S(n)$ such that $au+bv=0$. As $a\neq 0$, there exists a prime divisor $p$ of $n$ such that $a^{(p)}\neq 0$. Let $k=v_p(n)$ and $(u_p,v_p)$ be the image of $S$ under $f_{n,\,p^k}$. 

It follows that the sequence $(u_p,v_p)$ in $\mathbb Z_{p^k}$ has an $S(p^k)$-weighted zero-sum subsequence. As the image of $S=(u,v)$ under $f_{n,\,p}$ is the same as the image of $(u_p,v_p)$ under $f_{p^k,\,p}$, it follows that $(u_p,v_p)$ is a sequence in $\mathbb Z_{p^k}$ whose image under $f_{p^k,\,p}$ does not have any $Q_p$-weighted zero-sum subsequence. So by Lemma \ref{lboo} we get the contradiction that the sequence $(u_p,v_p)$ does not have any $S(p^k)$-weighted zero-sum subsequence.

Thus, we see that $S$ does not have any $S(n)$-weighted zero-sum subsequence. Hence, it follows that $D_{S(n)}\geq 3$.  
\end{proof}

\begin{theorem}\label{lbe2o}
We have that $D_{S(n)}\geq 3$ when $v_2(n)$ is even and at least two.
\end{theorem}

\begin{proof}
By two results by Mondal et al.\ \cite[Cor.\ 2, Lem.\ 7]{SKS2} and by the Chinese remainder theorem, we can find a sequence $S=(v_1,v_2)$ in $U(n)$ by a similar method as in Proposition \ref{seq} such that for each odd prime divisor $p$ of $n$ the sequence the image of $S$ under $f_{n,\,p}$ is not a $Q_p$-weighted zero-sum sequence and the image of $S$ under $f_{n,\,4}$ is $(1,1)$.

Suppose $T$ is an $S(n)$-weighted zero-sum subsequence of $S$. As the terms of $S$ are in $U(n)$, we see that $T$ must be $S$. Thus, there exists $a,b\in S(n)$ such that $au+bv=0$. As $a\neq 0$, there is a prime divisor $q$ of $n$ such that $a^{(q)}\neq 0$. We now use a similar argument as in the proof of Theorem \ref{lbo}, where we use Lemma \ref{les2} in addition to Lemma \ref{lboo}.
\end{proof}

\begin{theorem}\label{ubo'}
Let $p$ be an odd prime and $r$ be odd. Then we have $C_{S(p^r)}\leq 3$.
\end{theorem}

\begin{proof}
Let $S=(x,y,z)$ be a sequence in $\mathbb Z_{p^r}$ and let $S'$ be the image of $S$ under $f_{p^r,\,p}$. Mondal et al.\ \cite[Thm.\ 5]{SKS} showed that for an odd prime $p$ we have $C_{Q_p}=3$. Thus, we can find a subsequence $T$ whose terms are consecutive terms of $S$ such that the image of $T$ under $f_{p^r,\,p}$ is a $Q_p$-weighted zero-sum subsequence of $S'$. So by Lemma \ref{q2a} we see that $T$ is an $S(p^r)$-weighted zero-sum sequence. Hence, it follows that $C_{S(p^r)}\leq 3$.
\end{proof}

\begin{corollary}
Suppose $n$ is not a square and $v_2(n)$ is a non-negative, even integer. Then we have that $D_{S(n)}=C_{S(n)}=3$.
\end{corollary}

\begin{proof}
By Theorem \ref{lbe2o} we have $D_{S(n)}\geq 3$. From the assumptions on $n$, we see that there is an odd prime divisor $p$ of $n$ such that $v_p(n)$ is odd. Thus, by Corollary \ref{pr} and Theorem \ref{ubo'} we have that $C_{S(n)}\leq 3$. As $D_A(n)\leq C_A(n)$ for every $A\subseteq \mathbb Z_n$, it follows that $D_{S(n)}=C_{S(n)}=3$.
\end{proof}

\begin{theorem}\label{ubo2'}
We have that $C_{S(2^r)}\leq 2$ where $r$ is an odd number.  
\end{theorem}

\begin{proof}
Let $S=(x,y)$ be a sequence in $\mathbb Z_{2^r}$ and $S'=(x',y')$ be the image of $S$ under $f_{2^r,\,2}$. We can find a subsequence $T$ of $S$ such that the image of $T$ under $f_{2^r,\,2}$ is a zero-sum sequence. So by Lemma \ref{q2a} we see that $T$ is an $S(2^r)$-weighted zero-sum sequence. Hence, it follows that $C_{S(2^r)}\leq 2$.
\end{proof}

\begin{corollary} 
Suppose $n$ is an even positive integer such that $v_2(n)$ is odd. Then we have that $D_{S(n)}=C_{S(n)}=2$. 
\end{corollary}

\begin{proof}
It is easy to see that $D_{S(n)}\geq 2$. From Corollary \ref{pr} and Theorem \ref{ubo2'}, we get that $C_{S(n)}\leq 2$. For every $A\subseteq \mathbb Z_n$ as $D_A(n)\leq C_A(n)$, it follows that $D_{S(n)}=C_{S(n)}=2$.
\end{proof}

\section{\texorpdfstring{$D_{S(n)}$}{} when \texorpdfstring{$n$}{} is an odd square}

\begin{lemma}\label{lbee}
Let $p$ be a prime and $r$ be a non-zero even number. Suppose $(w_1,w_2)$ is a sequence in $U(p^r)$ whose image under $f_{p^r,\,p}$ is not a $Q_p$-weighted zero-sum sequence. Let $S=(u\,w_1,u\,w_2,p\,w_1,p\,w_2)$ where $u\in U(p^r)$. Then the sequence $S$ in $\mathbb Z_{p^r}$ does not have any $S(p^r)$-weighted zero-sum subsequence.
\end{lemma}

\begin{proof}
Suppose $T$ is an $S(p^r)$-weighted zero-sum subsequence of $S=(u\,w_1,u\,w_2,p\,w_1,p\,w_2)$. Let $I=\{i\in [1,2]:u\,w_i~\textrm{is a term of}~T\}$ and $J=\{j\in [1,2]:p\,w_j~\textrm{is a term of}~T\}$. As $T$ is an $S(p^r)$-weighted zero-sum sequence, for each $i\in I$ there exists $a_i\in S(p^r)$ and for each $j\in J$ there exists $b_j\in S(p^r)$ such that $\sum_{i\in I}a_iu\,w_i+\sum_{j\in J}b_j\,p\,w_j=0$. From Lemma \ref{sqpr}, for each $i\in I$ we have $a_i=p^{r_i}u_i$ for some even $r_i<r$ and $u_i\in U(p^r)^2$ and for each $j\in J$ we have $b_j=p^{s_j}v_j$ for some even $s_j<r$ and $v_j\in U(p^r)^2$. So we have
\begin{equation}\label{zs}
u\sum_{i\in I}p^{r_i}u_i\,w_i+\sum_{j\in J}p^{s_j+1}v_jw_j=0.
\end{equation}

Consider the set $L=\{r_i:i\in I\}\cup\{s_j+1:j\in J\}$. Let $r'$ be the minimum of $L$. As $r$ is even and $s_j$ is even for each $j\in J$, it follows that $r'\leq r-1$. Observe that $\{r_i:i\in I\}\cap \{s_j+1:j\in J\}=\emptyset$ as the $r_i$'s and $s_j$'s are even. Suppose there exists $i\in I$ such that $r_i=r'$. We claim that $I=\{1,2\}$ and $r_1=r_2=r'$. If not, from (\ref{zs}) we get that $p^{r'+1}$ divides $p^{r'}w$ where $w$ is a unit. As $r'\leq r-1$ we get the contradiction that $p$ divides $w$. By a similar argument if there exists $j\in J$ such that $s_j+1=r'$, then $J=\{1,2\}$ and $s_1+1=s_2+1=r'$.

Suppose $I=\{1,2\}$ and $r_1=r_2=r'$. As $r'\leq r-1$, from (\ref{zs}) we see that $p$ divides $u\,(u_1w_1+u_2w_2)$. As $u\in U(p^r)$, it follows that $f(u)\in U(p)$ where $f=f_{p^r,\,p}$. So we get that $f(u_1)f(w_1)+f(u_2)f(w_2)=0$. As $u_1,u_2\in U(p^r)^2$, it follows that $f(u_1),f(u_2)\in Q_p$. Thus, we get the contradiction that the image of the sequence $(w_1,w_2)$ under $f_{p^r,\,p}$ is a $Q_p$-weighted zero-sum sequence. We will get the same contradiction if $J=\{1,2\}$ and $s_1+1=s_2+1=r'$. Thus, it follows that $S$ does not have any $S(p^r)$-weighted zero-sum subsequence.
\end{proof}

\begin{theorem}\label{lbe}
Let $n$ be an odd square. Then we have that $D_{S(n)}\geq 5$. 
\end{theorem}

\begin{proof}
Let $m$ be the radical of $n$, i.e., the largest squarefree divisor of $n$. By Proposition \ref{seq} there exists a sequence $(u,v)$ in $U(n)$ such that for each prime divisor $p$ of $n$, the image of the sequence $(u,v)$ under $f_{n,\,p}$ does not have any $Q_p$-weighted zero-sum subsequence. Let $S=(u,v,mu,mv)$. We claim that the sequence $S$ in $\mathbb Z_n$ does not have any $S(n)$-weighted zero-sum subsequence from which it follows that $D_{S(n)}\geq 5$.

Suppose $T$ is an $S(n)$-weighted zero-sum subsequence of $S$. Let $x$ be a term of $T$ and $a\in S(n)$ be the coefficient of $x$ in an $S(n)$-weighted zero-sum which we obtain from $T$. As $a\neq 0$, there exists a prime divisor $p$ of $n$ such that $a^{(p)}\neq 0$. Let $r=v_p(n)$. It follows that the sequence $S^{(p)}$ has an $S(p^r)$-weighted zero-sum subsequence.

The image of the sequence $(u^{(p)},v^{(p)})$ under $f_{p^r,\,p}$ does not have any $Q_p$-weighted zero-sum subsequence. As $m$ is the largest squarefree divisor of $n$, it follows that $m^{(p)}=p\,w$ where $w\in U(p^r)$. It follows that $S^{(p)}$ is a sequence in $\mathbb Z_{p^r}$ which has the form as in the statement of Lemma \ref{lbee}. As $n$ is a square, we see that $r$ is a non-zero even number. So by Lemma \ref{lbee} we arrive at the contradiction that the sequence $S^{(p)}$ does not have any $S(p^r)$-weighted zero-sum subsequence. Hence, our claim must be true.
\end{proof}

We get the next result from the proof of \cite[Thm.\ 7]{ADU}.

\begin{lemma}\label{cm3}
Let $p$ be an odd prime, $r$ be a positive integer, and $A= U(p^r)^2$. Suppose we are given $y_1,y_2,y_3\in U(p^r)$. Then we have that \[Ay_1+(Ay_2\cup\{0\}\,)+(Ay_3\cup\{0\}\,)=\mathbb Z_{p^r}.\]
\end{lemma}

\begin{theorem}\label{lbd}
Let $r$ be a non-zero even number and $p$ be an odd prime. Then we have that $D_{S(p^r)}\leq 5$. 
\end{theorem}

\begin{proof}
Suppose $S=(x_1,\ldots,x_5)$ is a sequence in $\mathbb Z_{p^r}$. If $p^2$ divides some term $x_i$ of $S$, then $p^{r-2}x_i=0$. As $r$ is even, we see that $p^{r-2}=(p^{(r-2)/2})^2\in S(n)$. So we see that $(x_i)$ is an $S(p^r)$-weighted zero-sum subsequence of $S$ of length one. Thus, we may assume that $p^2$ does not divide any term of $S$.

It follows that each term of $S$ is either a unit or a multiple of $p$ by a unit. If at least three terms of $S$ are units or at least three terms of $S$ are of the form $pu$ where $u$ is a unit, by using Lemma \ref{cm3} we get an $S(p^r)$-weighted zero-sum subsequence of $S$. Thus, we see that $D_{S(p^r)}\leq 5$.
\end{proof}

\begin{corollary}
Let $n$ be an odd square. Then we have that $D_{S(n)}=5$.  
\end{corollary}

\begin{proof}
From Theorem \ref{lbe} we see that $D_{S(n)}\geq 5$ when $n$ is an odd square. Since every prime divisor $p$ of $n$ is odd and $v_p(n)$ is even, by using Corollary \ref{pr} and Theorem \ref{lbd} we see that $D_{S(n)}\leq 5$. Thus, it follows that $D_{S(n)}=5$.    
\end{proof}

\section{\texorpdfstring{$C_{S(n)}$}{} when \texorpdfstring{$n$}{} is an odd square}

\noindent
{\bf Notation:} 

If $T$ is a subsequence of $S$, then $S-T$ denotes the subsequence which is obtained by removing the terms of $T$ from $S$. The concatenation of the sequences $S-T$ and $T$ gives us a sequence whose terms are a permutation of the terms of the sequence $S$.

If $S$ is a sequence in $\mathbb Z_n$ and $d\in \mathbb Z_n$ such that all the terms of $S$ are divisible by $d$, then $S/d$ denotes the sequence in $\mathbb Z_n$ whose terms are obtained by dividing the corresponding terms of $S$ by $d$.

\begin{theorem}\label{25}
We have that $C_{U(25)^2}=9$. 
\end{theorem}

\begin{proof}
Let $S=(x_1,\ldots,x_9)$ be a sequence in $\mathbb Z_{25}$. We may assume that all the terms of $S$ are non-zero.

Suppose at least four terms of $S$ are units. From \cite[Lem.\ 2]{CM} it follows that $S$ is a $U(25)^2$-weighted zero-sum sequence. Let $S_1=(x_1,x_2,x_3)$, $S_2=(x_4,x_5,x_6)$, and $S_3=(x_7,x_8,x_9)$.

Suppose at most two terms of $S$ are units. Then we see that there exists $i\in [1,3]$ such that all the terms of $S_i$ are divisible by 5. Let $S_i'$ denote the sequence in $\mathbb Z_5$ which is the image of $S_i/5$ under $f_{25,5}$. From \cite[Thm.\ 4]{SKS} we have that $C_{Q_5}=3$. Thus, the sequence $S_i'$ has a $Q_5$-weighted zero-sum subsequence having consecutive terms. By \cite[Lem.\ 5]{SKS} it follows that the sequence $S_i$ (and hence the sequence $S$) has a $U(25)^2$-weighted zero-sum subsequence having consecutive terms.

So we may assume that exactly three terms of $S$ are units. If at least three consecutive terms of $S$ are non-units, by a similar argument as in the previous paragraph we get a $U(25)^2$-weighted zero-sum subsequence of $S$ having consecutive terms. So it follows that for each $i\in [1,3]$ there is exactly one term $y_i$ in the sequence $S_i$ which is a unit.

As $C_{Q_5}=3$ we see that the sequence $(y_1,y_2,y_3)$ has a subsequence $S_4$ having consecutive terms whose image $S_4'$ under $f_{25,5}$ is a $Q_5$-weighted zero-sum sequence. As $f_{25,5}$ is onto, it follows that there exists $k\in \mathbb Z_{25}$ such that a $U(25)^2$-weighted sum of the terms of $S_4$ is $-5k$. We will use this observation a bit later in this proof.

Let $J=\{i\in [1,3]:y_i$ is a term of $S_4\}$. Let $T$ be the subsequence of $S$ which is the concatenation of the sequences $S_i$ such that $i\in J$. It follows that $T$ is a subsequence of $S$ having consecutive terms. We claim that $T$ is a $U(25)^2$-weighted zero-sum sequence. Let $T_1=T-S_4$. As all the terms of $S$ are non-zero, all the terms of $T_1$ are of the form $5u$ where $u\in U(25)$. 

Let $T_1'$ denote the image of $T_1/5$ under $f_{25,5}$. Chintamani and Moriya \cite[Lem.\ 2]{CM} showed that $f_{25,5}(k)\in\mathbb Z_5$ is a $Q_5$-weighted sum of the terms of $T_1'$. Mondal et al.\ \cite[Lem.\ 5]{SKS} showed that $5k$ is a $U(25)^2$-weighted sum of the terms of $T_1$. As we have seen that $-5k$ is a $U(25)^2$-weighted sum of the terms of $S_4$, it follows that $T$ is a $U(25)^2$-weighted zero-sum sequence.

Thus, every sequence in $\mathbb Z_{25}$ having length nine has a $U(25)^2$-weighted zero-sum subsequence whose terms are consecutive terms of the given sequence. So it follows that $C_{U(25)^2}\leq 9$. Mondal et al.\ \cite[Cor.\ 5]{SKS} showed that $C_{U(25)^2}\geq 9$. Hence, it follows that $C_{U(25)^2}=9$.
\end{proof}

\begin{theorem}\label{osub}
Let $n$ be an odd square. Then we have that $C_{S(n)}\leq 9$. 
\end{theorem}

\begin{proof}
Mondal et al.\ \cite[Cor.\ 6]{SKS} showed that $C_{U(p^2)^2}=9$ when $p$ is a prime which is at least seven. From Remark \ref{c} we see that $C_{U(9)^2}\leq 9$ and from Theorem \ref{25} we see that  $C_{U(25)^2}\leq 9$. Thus, from Corollaries \ref{pr} and \ref{u2s} it follows that $C_{S(n)}\leq 9$.
\end{proof}

Chintamani and Moriya \cite[Lem.\ 1]{CM} showed the next result.

\begin{lemma}\label{cm'}
Let $p$ be a prime which is at least seven and $A=U(p^r)^2$. Then for every $x_1,x_2,x_3\in U(p^r)$ we have that   $Ax_1+Ax_2+Ax_3=\mathbb Z_{p^r}$.
\end{lemma}

Mondal et al.\ \cite[Lem.\ 7]{SKS} showed the next result, which follows easily from Lemma \ref{cm'}.

\begin{lemma}\label{cm}
Let $p$ be a prime which is at least seven and $S=(x_1,\ldots,x_k)$ be a sequence in $\mathbb Z_{p^r}$. Suppose at least three terms of $S$ are units. Then $S$ is a $U(p^r)^2$-weighted zero-sum sequence.
\end{lemma}

\begin{theorem}\label{ube'}
Let $p$ be a prime which is at least seven and $r$ be an even number which is at least four. Then we have that $C_{S(p^r)}\leq 5$.
\end{theorem}

\begin{proof}
Let $S=(x_1,\ldots,x_5)$ be a sequence in $\mathbb Z_{p^r}$. As $r$ is even, we see that $p^{r-2}=(p^{(r-2)/2})^2\in S(p^r)$. If $p^2$ divides some term $x$ of $S$, then it follows that $p^{r-2}x=0$ and so $S$ has an $S(p^r)$-weighted zero-sum subsequence of length one. Thus, we may assume that $p^2$ does not divide any term of $S$. So every term of $S$ is either a unit or of the form $p\,u$ where $u$ is a unit.

If at least three terms of $S$ are units, by Lemma \ref{cm} we see that $S$ is an $S(p^r)$-weighted zero-sum sequence. Thus, we may assume that at most two terms of $S$ are units. Then at least three terms of $S$ are of the form $p\,u$ where $u$ is a unit. We may assume that $x_1=pu_1,~x_2=pu_2,~x_3=pu_3$ where $u_1,u_2,u_3\in U(p^r)$. 

Consider the sequence $T=(u_1,u_2,u_3,px_4,px_5)$. By Lemma \ref{cm} we see that $T$ is a $U(p^r)^2$-weighted zero-sum sequence. So there exist $a_i$'s in $U(p^r)^2$ such that $a_1u_1+a_2u_2+a_3u_3+pa_4x_4+pa_5x_5=0$. Thus, it follows that  $a_1x_1+a_2x_2+a_3x_3+p^2a_4x_4+p^2a_5x_5=0$. As $r$ is at least four, we see that $p^2\neq 0$. Hence, it follows that $S$ is an $S(p^r)$-weighted zero-sum sequence. 
\end{proof}

\begin{corollary}\label{ube}
Let $n$ be an odd square which is divisible by $p^4$ where $p$ is a prime which is at least seven. Then we have that $C_{S(n)}=5$. 
\end{corollary}

\begin{proof}
As $n$ is a square, it follows that $v_p(n)$ is even. So from Corollary \ref{pr} and Theorem \ref{ube'} we have $C_{S(n)}\leq 5$. As we have that $D_A(n)\leq C_A(n)$ for every $A\subseteq\mathbb Z_n$, from Theorem \ref{lbe} it follows that $C_{S(n)}=5$.
\end{proof}
 
The next result follows easily from a result by Chintamani and Moriya \cite[Lem.\ 2]{CM}.

\begin{lemma}\label{cm5}
Let $r$ be a positive integer and $S$ be a sequence in $\mathbb Z_{5^r}$. Suppose at least four terms of $S$ are units. Then $S$ is a $U(5^r)^2$-weighted zero-sum sequence.
\end{lemma}

\begin{theorem}\label{ub7}
We have that $C_{S(5^r)}\leq 7$ when $r$ is an even number which is at least four.
\end{theorem}

\begin{proof}
We use a similar argument as in the proof of Theorem \ref{ube'}. The only change is that we replace Lemma \ref{cm} with Lemma \ref{cm5}.
\end{proof}

\begin{corollary}
Let $n$ be a square which is divisible by $5^4$. Then we have that $C_{S(n)}\leq 7$.
\end{corollary}

\begin{proof}
As $n$ is a square, it follows that $v_5(n)$ is even. Also, we have that $v_5(n)\geq 4$. So from Corollary \ref{pr} and Theorem \ref{ub7} we get that $C_{S(n)}\leq 7$.
\end{proof}

\begin{theorem}\label{ubos2}
Let $n$ be a square of an odd squarefree number. Then we have that $C_{S(n)}=9$. 
\end{theorem}

\begin{proof}
By Theorem \ref{osub} we get that $C_{S(n)}\leq 9$. We will construct a sequence $S$ of length eight in $\mathbb Z_n$ which has no $S(n)$-weighted zero-sum subsequence having consecutive terms. Hence, it will follow that $C_{S(n)}=9$.

By Proposition \ref{seq} there exists a sequence $S'=(u,v)$ in $U(n)$ such that for every prime divisor $p$ of $n$, the image $(u_p,v_p)$ of $S'$ under $f_{n,\,p}$ does not have any $Q_p$-weighted zero-sum subsequence. By the Chinese remainder theorem there exist $x,y\in \mathbb Z_n$ such that for each prime divisor $p$ of $n$ we have that $x^{(p)}=p\,u^{(p)}$ and $y^{(p)}=p\,v^{(p)}$. In this proof, for every $c\in \mathbb Z_n$ we will denote $f_{n,\,p}(c)$ by $c_p$. So it follows that $x_p=y_p=0$. Consider the sequence $S$ in $\mathbb Z_n$ defined as follows: \[S=(x,y,u,x,y,v,x,y).\]
Suppose there exists a subsequence $T$  of $S$ having consecutive terms which is an $S(n)$-weighted zero-sum sequence. 

\noindent
\texttt{Case 1:} 
Either $u$ or $v$ is a term of $T$.

Without loss of generality, we may assume that $u$ is a term of $T$.

Let $a\in S(n)$ be the coefficient of $u$ in the $S(n)$-weighted zero-sum which is obtained from $T$. As $a\neq 0$, there exists a prime divisor $p$ of $n$ such that $a^{(p)}\neq 0$. As $n$ is the square of a squarefree number, it follows that $v_p(n)=2$. So we see that $a^{(p)}\in S(p^2)$. As every non-zero term of $\mathbb Z_{p^2}$ is either a unit or a unit multiple of $p$, we see that $S(p^2)=U(p^2)^2$. As $a^{(p)}\in U(p^2)^2$, it follows that $f_{p^2,\,p}\big(a^{(p)}\big)\in Q_p$ and so $a_p\in Q_p$.

We claim that the sequence $(u_p,v_p)$ has a $Q_p$-weighted zero-sum subsequence.  
Suppose $v$ is a term of $T$. Let $b\in S(n)$ be the coefficient of $v$ in the $S(n)$-weighted zero-sum which is obtained from $T$. As $b\in S(n)$, it follows that $b_p\in Q_p\cup\{0\}$. So we get that $a_pu_p+b_pv_p=0$. If $v$ is not a term of $T$ then $a_pu_p=0$. This proves our claim from which we get a contradiction. 

\noindent
\texttt{Case 2:} 
Neither $u$ nor $v$ is a term of $T$.

As $T$ is a subsequence of consecutive terms, it follows that $T$ is a subsequence of the sequence $(x,y)$. Suppose $x$ is a term of $T$. Let $a\in S(n)$ be the coefficient of $x$ in the $S(n)$-weighted zero-sum which is obtained from $T$. By a similar argument as in the previous case, we see that there is a prime divisor $p$ of $n$ such that $a_p\in Q_p$. We claim that the sequence $(u_p,v_p)$ has a $Q_p$-weighted zero-sum subsequence. 

Suppose $y$ is a term of $T$. Then there exists $b\in S(n)$ such that $ax+by=0$. As $b\in S(n)$, it follows that $b_p\in Q_p\cup\{0\}$. As $ax+by=0$, it follows that $a^{(p)}x^{(p)}+b^{(p)}y^{(p)}=0$ in $\mathbb Z_{p^2}$. Thus, we get that $p\,\big(a^{(p)}u^{(p)}+b^{(p)}v^{(p)}\big)=0$ and so $a^{(p)}u^{(p)}+b^{(p)}v^{(p)}$ is divisible by $p$. Hence, it follows that $a_pu_p+b_pv_p=0$. By a similar argument, we see that if $y$ is not a term of $T$ then $a_pu_p=0$. This proves our claim from which we get a contradiction.

So we see that the sequence $S$ does not have any $S(n)$-weighted zero-sum subsequence having consecutive terms. 
\end{proof}

\section{Concluding remarks}

We have been unable to determine the constants $C_{S(3^r)}$ and $C_{S(5^r)}$ where $r$ is an even number which is at least four. For every such $r$ we have shown that $C_{S(3^r)}\in [5,9]$ and $C_{S(5^r)}\in [5,7]$. If the values of these constants are known, we can determine the value of $C_{S(n)}$ for every $n$.

We can try to characterize the sequences in $\mathbb Z_n$ of length $C_{S(n)}-1$ which do not have any $S(n)$-weighted zero-sum subsequence having consecutive terms. We can also try to characterize sequences in $\mathbb Z_n$ of length $D_{S(n)}-1$ which do not have any $S(n)$-weighted zero-sum subsequence.

\section{Acknowledgments}

\bigskip
\hrule
\bigskip

\noindent 2020 {\it Mathematics Subject Classification}: Primary 11B50; Secondary 11B75.

\noindent \emph{Keywords:} Davenport constant, weighted zero-sum sequence.

\bigskip
\hrule
\bigskip

%
%

\end{document}